\documentclass[12pt]{amsart}

\RequirePackage{amsmath, soul} 
\RequirePackage{amssymb} 
\RequirePackage{latexsym} 
\RequirePackage[mathscr]{eucal} 
\RequirePackage{verbatim} 
\RequirePackage{enumerate} 
\RequirePackage{xspace}
\RequirePackage[colorlinks]{hyperref} 
\RequirePackage[all]{xy}
\RequirePackage{color}
\RequirePackage{graphicx}
\newtheorem{theorem}{Theorem}[section]

\newtheorem{corollary}[theorem]{Corollary}

\newtheorem{definition}[theorem]{Definition}
\newtheorem{example}[theorem]{Example}

\newtheorem{lemma}[theorem]{Lemma}

\newcommand{\mb}[1]{\mathbf{ #1}}
\newcommand{\mc}[1]{\mathcal{ #1}}

\newcommand{\te}{\theta}

\newcommand{\3}{\mathbf 3}
\newcommand{\ty}[1]{\langle #1\rangle}

\newcommand{\ov}{\overline}

\newcommand{\dgr}{^{\dagger}}
\newcommand{\upr}{\mathord\uparrow}
\newcommand{\dnr}{\mathord\downarrow}
\newcommand{\cmp}{^{\star}}
\newcommand{\leqnd}{\leq_{\land}}
\newcommand{\leqr}{\leq_{\lor}}
\newcommand{\tsp}{2space$^{\star}$ }

\title[Stone-type representations of bisemilattices]{Stone-type representations and dualities for varieties of bisemilattices}
\author[Antonio Ledda]{Antonio Ledda}
\address{Antonio Ledda, Universit\`a di Cagliari, via Is Mirrionis 1, 09123, Cagliari,
Italia}
\email{antonio.ledda@unica.it}

\date{\today}
\begin{document}
\begin{abstract}
In this article we will focus our attention on the variety of distributive bisemilattices and some linguistic expansions thereof: bounded, De Morgan,  and involutive bisemilattices. After extending Balbes' representation theorem to bounded, De Morgan, and involutive bisemilattices, we make use of Hartonas-Dunn duality and introduce the categories of 2spaces and 2spaces$^{\star}$. The categories of 2spaces and 2spaces$^{\star}$ will play with respect to the categories of distributive bisemilattices and De Morgan bisemilattices, respectively, a role analogous to the category of Stone spaces with respect to the category of Boolean algebras. Actually, the aim of this work is to  show that these categories are, in fact, dually equivalent.
\end{abstract}

\maketitle

\section{Introduction}
It is a long-dated result, due to M. Stone \cite{St36}, that the theory of Boolean algebras (the algebraic counterpart of classical logic) can be framed within the theory of sets.  Stone's representation theorem for Boolean algebras has inspired an ample supply of dualities between appropriate categories of topological spaces and categories of partially ordered sets. Among them we may mention, e.g., Stone's duality for Boolean algebras \cite{St37}, Priestley's duality for distributive lattices \cite{Pr70, Pr72}, Hartonas-Dunn's duality for semilattices and lattices \cite{HD97}. \\

Bisemilattices --- or quasilattices, according to \cite{Pl67} --- were introduced by J. P\l onka as a common generalization of semilattices and lattices. Then, it was Padmanabhan who first introduced the name bisemilattices for these structures. Since this terminology has become standard over the years, in the present work we will adopt the name bisemilattices. Over the last 50 years, these algebraic structures have been the object of study of a considerable number of scholars. Among them we may mention R. Balbes, J. A. Kalman, A. Knoebel, R. McKenzie, J. P\l onka, and A. Romanowska \cite{Ba70, Ka, Kn, MRo, Pl67, GR91, Ro80, Ro80b, RoS81, Ro82}. \\

In this work we will focus our attention on a prominent class of bisemilattices, the variety of \emph{distributive bisemilattices} \cite{Pl67, MRo} and some remarkable linguistic expansions thereof: \emph{bounded distributive bisemilattices} \cite{FG93}, \emph{distributive De Morgan bisemilattices} \cite{Br00}, and \emph{involutive bisemilattices} \cite{GPBP}.\footnote{In fact, several authors (e.g. R. Balbes) consider bisemilattices as structures in the language $\{\land, \lor\}$ (see page \pageref{def:bsmlttc}). In general the involution operator $'$ and the constants $0,1$ need not be necessarily present.} Let us remark that, in general, bisemilattices are not assumed to be distributive, as in, e.g., \cite{Br00}. However, distributivity will be of crucial importance for our arguments, and it will be assumed throughout the development of our discourse. In fact, our starting point will be the class of distributive bisemilattices: the variety generated by the three-element algebra in Example \ref{exmp:3red}, which is the algebraic structure arising from the bisemilattice reduct of the tables of the \emph{weak propositional connectives} (introduced by Kleene in \cite{Kl52}). For a wider account on the logical aspects we refer the reader to section \ref{sec:motiv}.

As we mentioned at the beginning of this introduction, Stone's representation theorem for Boolean algebras has triggered a wide supply of results that build bridges between certain categories of lattice-based structures, and certain categories of topological spaces. From our perspective these connections are interesting, in general, because they represent a ``Rosetta's Stone'' deciphering a trilingual text: logical, algebraic, and topological. Perhaps, this sort of bridges may bring new tools for tackling old problems and shine a new light on well known concepts. A prominent example is in computer science, where the theory of domains provides a mathematical foundation for the semantics of programming languages (a comprehensive account can be found in \cite{GHKL}). As well as defining domains, D. Scott showed that they can be turned into a topological setting. And then M. Smyth observed that this is not just a technical trick. Scott topology captures a fundamental concept: open sets correspond to semi-decidable properties. This outstanding achievement paved the way for open sets to a new life, independent from the points of the original topological space, and led to a connection with an apparently unrelated field of mathematics: locale theory.

{In this article we discuss some dualities for distributive bisemilattices, bounded distributive bisemilattices, distributive De Morgan bisemilattices and involutive bisemilattices \cite{AA, Be, Br02, GPBP, GWW, GWW1, Ko, Ma, MoPi}, the algebraic alter-ego of Bochvar's internal logic and paraconsistent weak Kleene's logic (see section \ref{sec:motiv}). However, before explaining the details of our strategy, it is perhaps worth mentioning that}, in the nineties, Gierz and Romanowska \cite{GR91} (the interested reader may also consult the work of Romanowska and Smith \cite{RS96}) established a duality between distributive bisemilattices on the one hand, and compact totally disconnected partially ordered left normal bands on the other hand. Their results make essential use of the techniques developed in the theory of natural dualities \cite{CD}. The starting point is a ``schizophrenic'' object $H$ that lives in both categories (in the category in question and in its dual), and then consider the contravariant functors $\mathtt{Hom}(\_, H)$ to obtain a duality between the two categories.

In the present work, our first inspiration will be Balbes' set theoretical representation of distributive bisemilattices \cite[Thm. 4]{Ba70}: a representation theorem for distributive bisemilattices which parallels the representation of distributive lattices as rings of sets. After extending Balbes' result to bounded distributive bisemilattices, distributive De Morgan bisemilattices, and involutive bisemilattices (for these notions we refer the reader to section \ref{sec:bsml}), we take advantage of several ideas from Goldblatt \cite{Go75} and Hartonas-Dunn \cite{HD97} (in particular the notion of Fspace) and, in sections \ref{sec:top} and \ref{sec:demor}, we introduce the categories of 2spaces and 2spaces$^{\star}$. These categories will play with respect to the categories of distributive bisemilattices and distributive De Morgan bisemilattices, respectively, a r\^ole analogous to the one played by the category of Stone spaces with respect to the category of Boolean algebras. We will show, in fact, that these categories are dually equivalent. 

More precisely, we take advantage of Hartonas and Dunn's results on semilattices \cite{HD97} and use their notion of Fspace (see Definition \ref{def:mFsp}) for introducing the concept of 2space, which generalises their  $\bot$-frames (see Definition \ref{def:frame}) to a context in which the bijective correspondence between the Fspaces involved need not be in general a dual isomorphism, but could be as arbitrary as possible.\\

In order to keep this article self-contained, in section \ref{sec:bsml}, all the notions required from the theory of bisemilattices are dispatched. In section \ref{sec:motiv}, we propose a discussion on their logical significance. In section \ref{sec:rep}, we discuss Balbes' representation theorem and its extensions to the cases of bounded distributive, De Morgan distributive, and involutive bisemilattices. In section \ref{sec: fspcs}, we introduce Fspaces, the topological alter-egos of meet and join semilattices.
In section \ref{sec:top}, putting together the notion of Fspace with the representation theorem from section \ref{sec:rep}, a duality between the category of distributive bisemilattices and the category of 2spaces is established. Finally, in section \ref{sec:demor}, we enrich the notion of 2space and introduce the concept of 2space$^{\star}$, and we extend the duality to the categories of distributive De Morgan and involutive bisemilattices.

\section{Bisemilattices}\label{sec:bsml}
As we mentioned in the introduction, bisemilattices were firstly considered by J. P\l onka as a common generalization of semilattices and lattices, and, over the years, they have kindled the attention of several scholars.
An extensive guide to the bibliography on semirings --- of which distributive bisemilattices form a prominent subvariety --- can be found in K. Glazek's book \cite{Gl02} (for recent developments the interested reader may also consult \cite{BCL, Ha16a, Ha16b}).\\

\label{def:bsmlttc}In particular, a bisemilattice is an algebra $\mb A = \langle A, \land, \lor\rangle$ of type $\langle 2,2\rangle$ such that the reducts $\langle A, \land\rangle$ and $\langle A, \lor\rangle$ are semilattices. 
It is called \emph{distributive} in case $\land$ distributes over $\lor$, and $\lor$ distributes over $\land$. 
From an intuitive point of view, a (distributive) bisemilattice `is almost a (distributive) lattice' except for the fact that absorption's identities need not be satisfied. This supposedly mild failure strikes a major difference between lattices and bisemilattices. Indeed, if in the former case the orders induced by meet and join coincide, this is no longer so in the latter case. In principle, the partial orders induced by $\land$ and $\lor$ can be as arbitrary as possible, bearing no relation with each other.

We denote the variety of distributive bisemilattices by $\mc{DBS}$. For a bisemilattice $\mb A = \langle A, \land, \lor\rangle$, sometimes, we will adopt the notation $\ty{A, \leq_{\land}}$ and $\ty{A, \leq_{\lor}}$ to refer to the partially ordered sets naturally associated to the semilattice reducts $\langle A, \land\rangle$ and $\langle A, \lor\rangle$, respectively. All these observations are compactly subsumed in the following result, which is part of the folklore on the subject:
\begin{theorem}
There is a one-to-one correspondence between the variety of distributive bisemilattices and the class of bi-relational systems of the form $\ty{A,\leq_{\land},\leq_{\lor}}$, where $\ty{A,\leq_{\land},\land}$, $\ty{A,\leq_{\lor},\lor}$ are meet and join semilattices, respectively, that satisfy:
\begin{eqnarray*}
x\land(y\lor z)&\approx&(x\land y)\lor (x\land z);\\
x\lor(y\land z)&\approx&(x\lor y)\land (x\lor z).
\end{eqnarray*}
\end{theorem}
For our discourse, a relevant example of a distributive bisemilattice (in fact, it generates $\mc{DBS}$ as a variety) is the algebra that arises from the bisemilattice reduct of the tables of the binary \emph{weak propositional connectives}, discussed by Kleene in his ``Introduction to Metamathematics'' \cite{Kl52} (for further detail on the logical aspects, we refer the reader to section \ref{sec:motiv}).
\begin{example}\label{exmp:3red}{\rm
Consider the set $\{0,\frac12,1\}$ and define the following operations:
{
 \begin{align*}
&\qquad \begin{tabular}{c|ccc}
$\land$&$\mb0$&$\mb{1/2}$&$\mb1$\\
\hline$\mb0$&$0$&$1/2$&$0$\\
$\mb{1/2}$&$1/2$&$1/2$&$1/2$\\
$\mb{1}$&$0$&$1/2$&$1$
\end{tabular}
&\qquad \begin{tabular}{c|ccc}
$\lor$&$\mb0$&$\mb{1/2}$&$\mb1$\\
\hline$\mb0$&$0$&$1/2$&$1$\\
$\mb{1/2}$&$1/2$&$1/2$&$1/2$\\
$\mb{1}$&$1$&$1/2$&$1$
\end{tabular}
\end{align*}
}
It is not difficult to verify that this algebra is a distributive bisemilattice, although it is not a lattice, as shown by the Hasse diagrams of $\leq_{\land}$ and $\leq_{\lor}$:

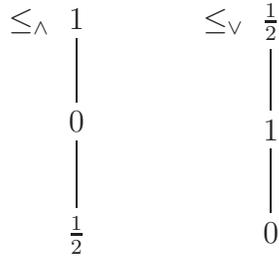
 \begin{eqnarray*}
\leq_{\land}  \xymatrix{
 1 \ar@{-}[d]&\\
0 \ar@{-}[d]&\\  
 \frac12 \ar@{-}&}
&\leq_{\lor}  \xymatrix{
\frac 12 \ar@{-}[d]&\\
1 \ar@{-}[d]&\\  
 0 \ar@{-}&}
 \end{eqnarray*}

In passing, let us observe that the algebra in this example is the P\l onka sum of the 2-element lattice and the one-element lattice.}
\end{example}

For the development of the ideas in this article, the notions of filter and ideal will be essential (see \cite{Ba70}).
\begin{definition}\label{def:fil.id}
A \emph{filter} in a distributive bisemilattice $\mb A$ is a non-void set $F\subseteq A$ such that, for $a,b\in A$:
\begin{enumerate}[(i)]
\item if $a\leq_{\land}b$ and $a\in F$, then $b\in F$;
\item if $a,\,b\in F$, then $a\land b\in F$.
\end{enumerate}
A filter $F\subsetneq A$ is \emph{prime} if, whenever $a\lor b\in F$, then either $a\in F$ or $b\in F$.
\end{definition}
The concepts of \emph{ideal} and \emph{prime ideal} are defined dually: $\leqnd,\,\land$ and $\lor$ are replaced by $\leqr,\,\lor$ and $\land$, respectively. If $\mb A$ is a distributive bisemilattice, we denote by $\mc F(A)$ ($\mc{F}p(A)$) and $\mc I(A)$ ($\mc{I}p(A)$) , the sets of (prime) filters and (prime) ideals of $\mb A$, respectively. The following result on filters, and its dual form for ideals, to be credited to R. Balbes \cite[Lemma 3]{Ba70}, will be extensively used.

\begin{theorem}\label{thm:fil-sep}
If $\mb A$ is a distributive bisemilattice, and, for $a,\,b\in A$, $a\not\leqnd b$, then there is a prime filter $F$ such that $a\in F$ and $b\notin F$.
\end{theorem}
Sometimes, we will refer to Theorem \ref{thm:fil-sep} as \emph{the prime filter} (\emph{ideal}) \emph{separation theorem}.\\

Semilattices can be regarded as the subvariety (denoted by $\mc{SEM}$) of $\mc {DBS}$ defined by the identity 
\begin{equation}\label{eq:SEM}
 x \land y  \approx x\lor y.
\end{equation}

Of course, in $\mc {SEM}$, both distributive laws are trivially satisfied. In case the semilattice reducts $\langle A, \land\rangle$ and $\langle A, \lor\rangle$ are commutative idempotent monoids -- i.e. they admit a greatest and smallest element $1$ and $0$, respectively -- we include these constants in the type and term the distributive bisemilattice $\mb A=\langle A, \land,\lor, 0,1\rangle$ \emph{bounded}. It can be readily seen that \emph{bounded distributive bisemilattices} form a variety (denoted by $\mc{BDBS}$) naturally specified by the equations 
\begin{equation}\label{eq:bnd}
x\land1\approx x\text{, and }x\lor0\approx x.
\end{equation}
Let us note that an application of the general theory of regularised varieties shows that bounded distributive bisemilattices are P\l onka sums of bounded distributive lattices over a semilattice with a neutral element $e$. For additional information the interested reader may consult \cite{Pl84} and \cite{Pl92}.\\
Whenever no danger of confusion is possible, in this work, by \emph{bounded bisemilattices we will mean bounded distributive bisemilattices}.\\

In 1993, Finn and Grigolia \cite{FG93}, and independently Brzozowski, seven years later, in \cite{Br00},  considered, under the name of \emph{De Morgan bisemilattices}, an expansion of bounded bisemilattices by an \emph{involution operation} $'$ that fulfils De Morgan laws:
\begin{equation}\label{eq:DDBS}
x'\lor y'\approx (x\land y)'\text{, and }x'\land y'\approx (x\lor y)'.
\end{equation}
We remark that, in general, De Morgan bisemilattices, as introduced in \cite{FG93} and \cite{Br00}, need not be distributive, while throughout this article \emph{distributivity will be always assumed}. The variety of \emph{bounded distributive bisemilattices satisfying De Morgan laws} will be denoted by $\mc{DDBS}$. Furthermore, let us notice that in \cite{FG93} and \cite{Br00}, it is also required that both semilattice reducts possess largest and smallest elements, call them $1_{\vee}, 0_{\vee}, 1_{\wedge}, 0_{\wedge}$, with the additional requirements that $1_{\vee} = 1_{\wedge}$ and $0_{\vee} = 0_{\wedge}$. According with Brzozowski's terminology, a bisemilattice with this feature is called \emph{consistently bounded}. This condition is rather strong. In particular, as shown by Brzozowski, a finite bisemilattice satisfying these conditions is a bilattice, i.e. both semilattice reducts are in fact lattices (for an extensive account on bilattices we refer the reader to \cite{F91, Da13}). Let us observe that, even if both these lattices are distributive, the whole bisemilattice need not necessarily be distributive (see Example 2 in \cite{Br00}).\\
From now on, whenever no danger of confusion is impending, \emph{by De Morgan bisemilattices} we will mean \emph{distributive bisemilattices that satisfy the De Morgan laws}. We remark that, throughout our discourse, \emph{the structures in $\mc{DDBS}$ are, in general, not assumed to be consistently bounded}, but bounded.

Recently, J. Gil F\'erez and his coauthors (see \cite{GPBP}), in the attempt of identifying a candidate suitable to play the r\^ole of an algebraic counterpart of paraconsistent weak Kleene logic \cite[\S64]{Kl52},\footnote{For a wider account on Kleene's logics and three-valued semantics, we refer the reader to section \ref{sec:motiv}.} introduced the stronger notion of \emph{involutive bisemilattices}.
An equational base for this variety, which will be denoted by $\mathcal{IDBS}$, follows:
\begin{enumerate}[({I}1)]
\item $x\lor x\approx x$;
\item $x\lor y\approx y\lor x$;
\item $x\lor (y\lor z)\approx (x\lor y)\lor z$;
\item $x''\approx x$;
\item $x\land y\approx (x'\lor y')'$;
\item $x\land(x'\lor y)\approx x\land y$;
\item $0\lor x\approx x$;
\item $0\approx 1'$.
\end{enumerate}
This concise axiomatization yields rather strong properties, such as full De Morgan laws, distributivity and boundness. These algebras can be regarded as De Morgan bisemilattices with the additional requirement that $x\land(x'\lor y)\approx x\land y$. Furthermore, we observe that, by virtue of axioms (I5) and (I8), the operations $\land$ and $1$ are entirely determined by $\lor$, $'$, and $0$, and vice versa. None the less, the orders that $\land$ and $\lor$ induce may still differ (cf. Example \ref{exmp:3}). In general, algebras in $\mc {IDBS}$ are not supposed to be consistently bounded. For a counterexample, we refer the reader again to Example \ref{exmp:3}.\\
Since no danger of confusion will be impending, from section \ref{sec:top} on, by $\mc{DBS}$, $\mc{DDBS}$ and $\mc{IDBS}$ we will denote, with a slight abuse of language, the categories whose objects are distributive, bounded distributive De Morgan, and involutive bisemilattices, and {whose arrows are the relative (eventually constant-preserving) homomorphisms}.\\
\section{Why a three-valued semantics?}\label{sec:motiv}
{In this section, we present and elaborate on few concepts that will be relevant to the development of our discourse. To begin with, we present the following, crucial, example:
\begin{example}\label{exmp:3}{\rm
$\3$ is the involutive bisemilattice with universe $\{0,\frac12,1\}$, on which the following operations are defined:
{
 \begin{align*}
\begin{tabular}{c|c}
$'$&\\
\hline$\mb0$&$1$\\
$\mb{1/2}$&$1/2$\\
$\mb1$&$0$
\end{tabular}
&\qquad \begin{tabular}{c|ccc}
$\land$&$\mb0$&$\mb{1/2}$&$\mb1$\\
\hline$\mb0$&$0$&$1/2$&$0$\\
$\mb{1/2}$&$1/2$&$1/2$&$1/2$\\
$\mb{1}$&$0$&$1/2$&$1$
\end{tabular}
&\qquad \begin{tabular}{c|ccc}
$\lor$&$\mb0$&$\mb{1/2}$&$\mb1$\\
\hline$\mb0$&$0$&$1/2$&$1$\\
$\mb{1/2}$&$1/2$&$1/2$&$1/2$\\
$\mb{1}$&$1$&$1/2$&$1$
\end{tabular}
\end{align*}
}}
\end{example}

As we mentioned on page \pageref{exmp:3red}, the algebra $\mb 3$ in Example \ref{exmp:3} (which generates $\mc{IDBS}$ as a variety) arises from the matrices of the \emph{weak propositional connectives}, discussed by Kleene in \cite{Kl52}, in which he distinguishes between a \emph{weak} and \emph{strong} sense of propositional connectives.
These tables are devised to describe situations in which partially defined / unknown properties (weak / strong sense of the connectives, respectively) are present. Indeed, Kleene's main assumption \cite{Kl38} is that there are statements whose logical truth or falsity is either not defined, or not available in terms of accessible algorithms, or it is not essential. 
In the long tradition of the subject, several interpretations of the possible meaning of the third value have been proposed.
In fact, beside Kleene's, there are many other perspectives that call for the introduction of a three-valued semantics.\footnote{An extensive account on three-valued semantics lies outside of the scope of this article, in which we confine ourselves in mentioning only perspectives that are strictly relevant to our discourse.}
\emph{Vagueness} is one of them. In natural language, words, in particular those that are used to detail ordinary objects, are \emph{vague}: it is not clear whether they fall completely under the scope of a property or not. In fact, these properties have a ``fuzzy border''.\footnote{An interesting account on a three-valued semantics for vagueness, as distinct from partially defined / unknown, is in H. Kamp \cite{Ka75}, in which supervaluations are considered. Vagueness is also the inspiration of a broad field of research known, nowadays, as fuzzy logic(s), intersecting philosophy, classical mathematics, and engineering (for an introduction to the subject the reader may consult \cite{Ha97}).}
 For example:
\begin{center}
Italy is boot shaped,\\
France is exagonal. 
\end{center}

A detailed account on the attempts to apply three-valued logic to the problem of vagueness, and the difficulties that many-valued logics and degrees of truth faced is in Williamson's book \cite[sections 4.3-4.6]{Wi}.

Another chief reference for the interpretation of a third logical value traces back to Aristotle's reflections on future contingents. Indeed, if \emph{tertium non datur} were true, the sentence
\begin{center}
`There will be a sea battle tomorrow'
\end{center}
will be either true or false. And therefore the future event it describes will be determined. But this clashes with the fact that matters that rely on human decision are not determined. Thus the excluded middle must be false. Future contingent sentences were, usually, classified as undetermined, rather than true or false. Discussions on this type of themes can be found also in medieval philosophy, that abounds in arguments on whether God's foreknowledge is consistent with human free willing.
Examples of this sort are interesting because the necessity of identifying an alternative to what is true or false emerges from a philosophical riddle, rather than a trivial attempt to classify examples of sentences on the fuzzy border between truth and falsity.\footnote{The literature on this subject is immense. For a preliminary discussion, the reader may consult \cite{KK}.}

Finally, we may recall the \emph{presupposition failure problem}.\footnote{A stimulating, albeit dated report on presupposition failure is by J. Martin \cite{Mar87}. More recent references are Cobreros et al. \cite{CERV} and van Eijck \cite{vEij}.} Questions, very often, presuppose facts. Indeed, an issue may not arise in case certain presuppositions / assumptions are not true. In fact, both the meaningful assertion or denial of a proposition may rely on the assumption of further premises, and, if those premises are false, any trial of asserting or denying the original sentence becomes meaningless.

The literature on presupposition is extensive. Perhaps, the first account on the subject is in Frege's ``Sense and Reference'', but the first explicit use of presupposition as a motivation for a three-valued evaluation of sentences was an application of Bochvar's tables (see below) to presupposition by Smiley in his ``Sense without denotation'' \cite{Sm60}.

Formally, a statement $A$ presupposes a set of premises $\{B_{i}\}_{i\in I}$, if, in any world $w$, if $A$ is true or false in $w$, then, for all $i\in I$, $B_{i}$ is true in $w$. This idea assumes that the evaluations true and false do not exhaust all possible eventualities, and that the excluded middle law, in general, is false. Two standard examples follow. Each of them consists of three sentences, the first is an assertion and the second is its denial, and both presuppose the third sentence of the triple. If the third statement is false, neither of the first two is true. Actually, the question whether they are true or false does not arise.
\begin{center}
\begin{enumerate}
\item John is a bachelor;
\item John is not a bachelor; 
\item  John exists.
\end{enumerate}
\end{center}
\begin{center}
\begin{enumerate}
\item The king of France is bald;
\item The king of France is not bald;
\item The king of France exists.
\end{enumerate}
\end{center}

\smallskip

In Kleene's three-valued logics, uncertainty / lack of knowledge of the applicability of a property is denoted by $\frac12$.\footnote{In Kleene's original notation the truth values are denoted by $\mathfrak{f,t}$ and $\mathfrak u$, which stands for \emph{undefined} \cite[p.332]{Kl52}} The matrices of Kleene's strong connectives are the following:
 \begin{align*}
\begin{tabular}{c|c}
$'$&\\
\hline$\mb0$&$1$\\
$\mb{1/2}$&$1/2$\\
$\mb1$&$0$
\end{tabular}
&\qquad \begin{tabular}{c|ccc}
$\land$&$\mb0$&$\mb{1/2}$&$\mb1$\\
\hline$\mb0$&$0$&$0$&$0$\\
$\mb{1/2}$&$0$&$1/2$&$1/2$\\
$\mb{1}$&$0$&$1/2$&$1$
\end{tabular}
&\qquad \begin{tabular}{c|ccc}
$\lor$&$\mb0$&$\mb{1/2}$&$\mb1$\\
\hline$\mb0$&$0$&$1/2$&$1$\\
$\mb{1/2}$&$1/2$&$1/2$&$1$\\
$\mb{1}$&$1$&$1$&$1$
\end{tabular}
\end{align*}

The family of ``Kleene logics'' split into two main subfamilies, according to which set of connectives (weak or strong) a logic refers to. In fact, to each family of connectives correspond two multiple-valued logics. On the one hand, as regards the strong Kleene connectives, we have \emph{the strong Kleene logic $\mb K_{3}$} \cite{Kl52} and \emph{the logic of paradox $\mb P_{3}$} \cite{Pri78}.\footnote{One may count within this set of logics also Bochvar's external logic, which corresponds to $\mb K_{3}$ \cite{Bo38}.} On the other, as regards the weak Kleene connectives, we have \emph{Bochvar's internal three-valued logic $\mb B_{3}$} (also known as \emph{Kleene's weak three-valued logic}) \cite{Bo38} and \emph{the paraconsistent weak Kleene logic $\mb{PWK}$} \cite{Ha49, Prio57}. The difference between the two logics within the two families lies precisely in the way tautologies are conceived. In $\mb K_{3}$ and $\mb B_{3}$ only ``$1$'' is designated -- a formula of the logic is a tautology iff it takes value to a designated truth value -- while in $\mb P_{3}$ and $\mb {PWK}$ both ``$\frac12,\,1$'' are designated. 

Perhaps, a moment's reflection on the possible significance of the value $\frac12$ within the family of Kleene's logics could be useful here.
The literature on strong Kleene's logic is vast. Kripke \cite{Kr75} finds this a convenient way to deal with paradoxical statements like ``What I'm now saying is false'', and uses $\mb K_{3}$ evaluations to propose a theory of truth which contains its own truth predicate.
An accurate and very sensible account on $\mb K_{3}$ is suggested by K\"orner in \cite{Ko66}, where the author introduces the notion of \emph{inexact class} of a non-empty domain generated by a partial definition of a property as a \emph{three-valued characteristic function}. K\"orner's interpretation has found applications in the eighties and nineties in the theory of rough sets \cite{Pa91}, and in the approximation logic based on it \cite{Ra86}.
Quoting Kleene, in his strong logic the truth values $0,1,\frac12$ can be interpreted, respectively, as 
\begin{quotation}
{``true'', ``false'', ``unkown (or value immaterial)''. Here unknown is a category into which we can regard any proposition as falling, whose value we either do not know or choose for the moment to disregard; and it does not then exclude the other two possibilities true or false.[...] Suppose that, for a given $x$, we know $Q(x)$ to be true. Then, using $1,0,\frac12$ as  ``true'', ``false'', ``unknown'' [...] we can conclude [...][cf. the tables of the strong connectives] that $Q(x)\lor R(x)$ is true. \cite[p. 335]{Kl52}} 
\end{quotation}

Instead, in Priest's logic,\footnote{This logic was introduced for the first time by Asenjo and Tamburino in \cite{AT}.} passionately supported by G. Priest in the framework of a dialetheic approach to the truth-theoretical and set-theoretical paradoxes, $\frac12$ can be interpreted as being ``overdetermined'', being \emph{both true and false}.\footnote{In Priest's article $1,0,\frac12$ are denoted by $t,f,p$, respectively \cite{Pri78}.}
\begin{quotation}
Classical logic errs in assuming that no sentence can be both true and false.We wish to correct this assumption. If a sentence is both true and false, let us call it ``paradoxical'' ($p$).[...] I will do just a couple of examples. If $A$ is $t$ and $B$ is $p$, then both $A$ and $B$ are true. Hence $A\land B$ is true. However, since $B$ is false, $A\land B$ is false. Thus $A\land B$ is paradoxical. If $A$ is $f$ and $B$ is $p$, then both $A$ and $B$ are false. Hence $A \land B$ is false. If $A \land B$ were true as well, then both $A$ and $B$ would be true, but $A$ is false only. Hence $A \land B$ is false.[...] Reasoning in a similar way we can justify the table of disjunction. \cite[pp. 226-227]{Pri78}
\end{quotation}
As regards Bochvar's internal logic \cite{Bo38} (for a detailed overview on Bochvar's internal/external logics, the reader may consult Malinowski's survey on many-valued logics \cite{Mal07}), the matrices of its connectives are devised following a principle similar to Kleene's standpoint: 
\begin{quotation}
every compound proposition including at least one meaningless component is meaningless, in other cases its value is determined classically. 
\end{quotation}
 As mentioned earlier, Bochvar considers only $1$ as designated value for his internal logic. It follows from this assumption that this logic has no tautologies. Indeed, a moment's reflection shows that, for any formula $A$, if all its propositional variables receive value $\frac12$, then $A$ itself will receive value $\frac12$. Ideas similar to Bochvar's have been pursued by several authors whose aim was providing logics that could express the notions of vagueness or non-sense. Among these attempts, it is perhaps worth mentioning Halld\'en's \emph{Non-sense logic} \cite{Ha49} (of which $\mb{PWK}$ is a linguistic reduct), whose matrices are the same as Bochvar's, but differ from internal logic in two main points: first it considers both $1$ and $\frac12$ as designated values, and, second, its language is enriched by a new unary connective $+$ which behaves as a sort of ``clarifier'' that expresses the meaningfulness of a proposition. In fact, for a formula $A$, $+A$ takes value $0$, if $A$ is meaningless; $1$ otherwise. Successive further elaborations on Non-sense logic are by \r Aqvist \cite{Aq}, Segerberg \cite{Sege}, and Omori, who, recently, tried to connect Halld\'en's Non-sense logic to the \emph{Logic of Formal Inconsistency} \cite{Om}.

Finally, we close this section by observing that if the matrices of strong Kleene's connectives define a distributive lattice order, the matrices of weak Kleene's connectives do not: they form a \emph{bounded distributive De Morgan bisemilattice}. 
}

\section{Balbes-type representation theorems}\label{sec:rep}
In 1970, in his article ``A representation theorem for distributive quasi-lattices'' \cite{Ba70}, R. Balbes provided a representation theorem for distributive bisemilattices which is analogous to the representation of distributive lattices as rings of sets. {In order to prove a duality between distributive bisemilattices (and expansions thereof)
and appropriate classes of topological spaces,} in this section we take full advantage of Balbes' theorem and widen its scope to encompass the cases of several expansions of distributive bisemilattices. Namely, De Morgan bisemilattices, bounded bisemilattices, and involutive bisemilattices. For readability sake, we will sketch, whenever it is insightful, up-to-date versions of Balbes' original arguments. We believe that this will be of some utility for the development of our discourse because Balbes' original result, which focuses on prime filters (ideals) only, \emph{works well also for of filters} (\emph{ideals}) \emph{in general}.\\

Our first move will be a set theoretical representation theorem for De Morgan bisemilattices, and then we will discuss how to adapt this representation to embrace the concept of involutive bisemilattices \cite{GPBP}, and bounded bisemilattices.\\
Consider two families of sets $X$ and $Y$ closed with respect to finite intersections and unions, respectively, such that $Y$ has a minimal element $\overline 0 $. Let $\theta: X\to Y$ be a bijective correspondence satisfying \\
{\small
\begin{equation}
\forall A,B,C\in X:\; A\cap\theta^{-1}(\theta(B)\cup\theta(C)) =\theta^{-1}(\theta(A\cap B)\cup\theta(A\cap C));\label{eq:bal1}
\end{equation}
\begin{equation}
\forall P,Q,R\in Y:\; P\cup\theta(\theta^{-1}(Q)\cap\theta^{-1}(R)) =\theta(\theta^{-1}(P\cup Q)\cap\theta^{-1}(P\cup Q));\label{eq:un}
\end{equation}
}
and let $\cmp:Y\to X$ be an order-dual isomorphism such that 
\[
^\star\circ\theta=\theta^{-1}\circ {^{\star}}^{-1}.
\]
Define, for $A,B\in X$, the following operations:
\begin{eqnarray*}\label{operations}
A+B&=&\theta^{-1}(\theta(A)\cup\theta(B));\\
A\cdot B&=&A\cap B;\\
A\dgr&=&(\theta{(A)})\cmp;\\
\bot&=&\theta^{-1}(\overline0);\\
\top&=&\overline0^\star.
\end{eqnarray*}

\begin{lemma}\label{lem:suggerito1}
The structure $\mb X=\langle X, \cdot,+,\dgr, \bot, \top\rangle$ is in $\mc{DDBS}$.
\end{lemma}
\begin{proof}
Following \cite[Thm. 4]{Ba70}, one obtains that $\mb X=\langle X, +,\cdot\rangle$ is a distributive bisemilattice. Since $^\star\circ\theta=\theta^{-1}\circ {^\star}^{-1}$, it can be readily verified that $\dgr$ is an involution.

Moreover, from the fact that $^{\star}$ is an order dual isomorphism, and because $\bot\dgr=(\theta^{-1}(\overline0))\dgr=(\theta\theta^{-1}(\overline0))^{\star}=\overline0^{\star}=X=\top$, we obtain that $A\cdot \top=A\cap\overline0^{\star}=A\cap X=A$, and $A+\bot=\theta^{-1}(\theta(A)\cup\theta\theta^{-1}(\overline0))=\theta^{-1}(\theta(A)\cup\overline0)=\theta^{-1}(\theta(A))=A$.

As regards De Morgan's laws: $(A+B)\dgr=(\theta^{-1}(\theta(A)\cup\theta(B)))\dgr=(\theta\theta^{-1}(\theta(A)\cup\theta(B)))^{\star}=(\theta(A))^{\star}\cap(\theta(B))^{\star}=(\theta(A))^{\star}\cdot(\theta(B))^{\star}=A\dgr\cdot B\dgr$. Upon noticing that the fact that $^{\star}$ is an order dual isomorphism implies, for all $A^{\star^{-1}}, B^{\star^{-1}}\in Y$, that: $(A^{\star^{-1}}\cup B^{\star^{-1}})=(A^{\star^{-1}}\cup B^{\star^{-1}})^{\star\star^{-1}}=(A^{\star^{-1}\star}\cap B^{\star^{-1}\star})^{\star^{-1}}=(A\cap B)^{\star^{-1}}$, and the requirement that $^\star\circ\theta=\theta^{-1}\circ {^\star}^{-1}$ implies that $\theta\circ^\star\circ\theta=\theta\circ\theta^{-1}\circ {^\star}^{-1}= {^\star}^{-1}$, we obtain that 
\begin{eqnarray*}
(A\cdot B)\dgr&=&(\theta(A\cdot B))^{\star}\\
&=&(\theta(A\cap B))^{\star}\\
&=&\theta^{-1}((A\cap B)^{{\star}^{-1}})\\
&=&\theta^{-1}(A^{{\star}^{-1}}\cup B^{{\star}^{-1}})\\
&=&\theta^{-1}\big(\theta((\theta(A))^{\star})\cup \theta((\theta(B))^{\star})\big)\\
&=&(A\dgr+B\dgr).
\end{eqnarray*}
\end{proof}
We will refer to a structure of the form $\mb X=\langle X, \cdot,+,\dgr, \bot, \top\rangle$ as a \emph{{DDBS} of sets.}
On the other hand, suppose that $\mathbf L$ is in $\mc{DDBS}$, and consider $\mathcal {F}(L),\mathcal {I}(L)$, the sets of filters and ideals of $\mathbf L$, respectively (cf. page \pageref{def:fil.id}). Define for $x\in L$:
\begin{eqnarray*}\label{upr-dnr-arrows}
\mathord{\uparrow} x=&\{F\in\mathcal {F}(L):x\in F\}; \\
\mathord\downarrow x=&\{I\in\mathcal {I}(L):x\notin I\}.
\end{eqnarray*}
Moreover, set 
\begin{eqnarray*}
X=&\{\mathord\uparrow x:x\in L\}; \\
Y=&\{\mathord\downarrow x:x\in L\}.
\end{eqnarray*}
and let $\theta:\mathord\uparrow x\mapsto \mathord\downarrow x$, $\theta^{-1}: \mathord\downarrow x\mapsto\mathord\uparrow x$, $^{\star}:\dnr x\mapsto\upr x'$.
\begin{theorem}\label{thm: rpr}
Let $\mathbf L=\langle L, \land ,\lor , ',0,1\rangle$ in $\mc{DDBS}$, then the structure $\mb X=\langle X, \cdot,+,\dgr, \bot, \top\rangle$ is a DDBS of sets isomorphic to $\mb L$.
\end{theorem}
\begin{proof}
 The fact that $\langle X,\cap, \mathord\uparrow1\rangle$ is a meet semilattice with greatest element $\mathord\uparrow1$ is immediate. 
By Theorem \ref{thm:fil-sep} and its dual version for ideals, the mappings $\theta:\mathord\uparrow x\mapsto \mathord\downarrow x$ and $\theta^{-1}: \mathord\downarrow x\mapsto\mathord\uparrow x$ are one-to-one, and a moment's reflection shows that they are mutually inverse correspondences between $X$ and $Y$.

If $H\in \mathord\downarrow x\cup\mathord\downarrow y$, then, without loss of generality, $x\notin H$, and since $x\leqr x\lor y$, $x\lor y\notin H$, i.e. $H\in \dnr (x\lor y)$. Conversely, suppose $H\notin \mathord\downarrow x\cup\mathord\downarrow y$. Then, $x,\,y\in H$, and so $x\lor y\in H$, and consequently $H\notin\mathord\downarrow(x\lor y)$. That is to say, $\mathord\downarrow(x\lor y)=\mathord\downarrow x\cup\mathord\downarrow y$, and so $Y$ is closed under $\cup$. Finally, it is straightforward to observe that $\mathord\downarrow0$ is the least element of the semilattice $\langle Y,\cup, \mathord\downarrow0\rangle$.

Let us note that the assignment defined for $x\in L$: $x\mapsto \upr x$, preserves finite meets. In fact, it is clear that $\mathord\uparrow(x\land  y)\subseteq\mathord\uparrow(x)\cdot \mathord\uparrow(y)=\mathord\uparrow(x)\cap \mathord\uparrow(y)$, because $(x\land  y)\land  x=y\land( x\land  x)=x\land  y$, and analogously interchanging $x$ with $y$. So $x\land y\leqnd x,y$, and therefore, by Definition \ref{def:fil.id}-(i), if $F\in \mathord\uparrow(x\land y)$, then $F$ belongs to both $ \mathord\uparrow x$ and $ \mathord\uparrow y$.
Conversely, in case $H\in\mathord\uparrow(x)\cdot \mathord\uparrow(y)=\{F\in \mc F(L):x\in F\}\cap\{G\in \mc F(L):y\in G\}$, then, by Definition \ref{def:fil.id}-(ii), from the fact that $x,y\in H$, it follows that $x\land  y\in F$, and therefore our claim follows. As regards $\lor$, we have that $\upr(x\lor y)=\upr\dnr(x\lor y)=\upr(\dnr x\cup\dnr y)=\upr(\dnr \upr x\cup\dnr \upr y)=\upr x+\upr y$.
Furthermore, it can be shown that the assignment $x\mapsto \mathord\uparrow x$ is an isomorphism between the reducts $\langle \land ,\lor \rangle$ and $\langle +,\cdot\rangle$. In fact, injectivity is warranted by the prime filter separation (see Theorem \ref{thm:fil-sep}), i.e. for distinct $x,\,y\in L$, there is a prime filter $F$ of $\mb L$ such that $x\in F$ but $y\notin F$, and surjectivity is obvious. Also, we have already seen that $\mathord\uparrow x\cdot \mathord\uparrow y=\mathord\uparrow(x\land  y)$. Moreover, $\mathord\uparrow x+ \mathord\uparrow y=\mathord\uparrow(\mathord\downarrow\mathord\uparrow x\cup \mathord\downarrow\mathord\uparrow y)=\mathord\uparrow(\mathord\downarrow x\cup \mathord\downarrow y)=\mathord\uparrow \mathord\downarrow(x\lor   y)=\mathord\uparrow (x\lor   y)$. We now show that condition \eqref{eq:bal1} is satisfied:
\begin{eqnarray*}
\mathord\uparrow x\cap\theta^{-1}(\theta(\mathord\uparrow y)\cup\theta(\mathord\uparrow z))&=&\mathord\uparrow x\cap\mathord\uparrow(\mathord\downarrow y\cup\mathord\downarrow z)\\
&=&\mathord\uparrow x\cap\mathord\uparrow\mathord\downarrow ( y\lor z)\\
&=&\mathord\uparrow x\cap\mathord\uparrow( y\lor z)\\
&=&\mathord\uparrow x\cdot\mathord\uparrow( y\lor z)\\
&=&\mathord\uparrow( x\land( y\lor z))\\
&=&\mathord\uparrow( (x\land y)\lor (x\land z))\\
&=&\mathord\uparrow(\mathord\downarrow\mathord\uparrow(x\land y)\cup \mathord\downarrow\mathord\uparrow(x\land z))\\
&=&\mathord\uparrow(\mathord\downarrow(\mathord\uparrow x\cap\mathord\uparrow y)\cup \mathord\downarrow(\mathord\uparrow x\cap\mathord\uparrow z))\\
&=&\theta^{-1}(\theta(\mathord\uparrow x\cap\mathord\uparrow y)\cup \theta(\mathord\uparrow x\cap\mathord\uparrow z)).
\end{eqnarray*}
As regards condition \eqref{eq:un}, 
\begin{eqnarray*}
\mathord\downarrow x\cup\theta(\theta^{-1}(\dnr y)\cap\theta^{-1}(\dnr z))&=& \mathord\downarrow x\cup\mathord\downarrow(\mathord\uparrow \mathord\downarrow y\cap\mathord\uparrow\mathord\downarrow z)\\
&=&\mathord\downarrow x\cup\mathord\downarrow(\mathord\uparrow y\cap\mathord\uparrow z)\\
&=&\mathord\downarrow x\cup\mathord\downarrow(y\land z)\\
&=&\mathord\downarrow (x\lor(y\land z))\\
&=&\mathord\downarrow ((x\lor y)\land(x\lor z))\\
&=&\mathord\downarrow\mathord\uparrow ((x\lor y)\land(x\lor z))\\
&=&\mathord\downarrow (\mathord\uparrow(x\lor y)\cap\mathord\uparrow(x\lor z))\\
&=& \mathord\downarrow (\mathord\uparrow\mathord\downarrow(x\lor y)\cap\mathord\uparrow\mathord\downarrow(x\lor z))\\
&=&\mathord\downarrow (\mathord\uparrow(\mathord\downarrow x\cup\mathord\downarrow y)\cap\mathord\uparrow(\mathord\downarrow x\cup\mathord\downarrow z))\\
&=&\theta (\theta^{-1}(\mathord\downarrow x\cup\mathord\downarrow y)\cap\theta^{-1}(\mathord\downarrow x\cup\mathord\downarrow z)).
\end{eqnarray*}

Let us recall that, for $x\in L$, $(\dnr x)^{\star}=\{F\in \mathcal F:x'\in F\}=\upr x'$.

Since $'$ is an involution, it is not difficult to see that $^{\star}$ is a bijection between $Y$ and $X$. We now show that $^{\star}$ is an order dual mapping. Indeed, $(\dnr x\cup\dnr y)^{\star}=(\dnr(x\lor y))^{\star}=\upr((x\lor y)')=\upr(x'\land y')=\upr x'\cap\upr y'=(\dnr x)^{\star}\cap(\dnr y)^{\star}$. That $^{\star}\circ\theta=^{\star}\circ\dnr=\upr\circ{^{\star}}^{-1}=\theta^{-1}\circ{^{\star}}^{-1}$ follows from:
\begin{eqnarray*}
\left(\theta\left(\left(\theta\left(\upr x\right)\right)^{\star}\right)\right)^{\star}&=& \left(\dnr\left(\left(\dnr\left(\upr x\right)\right)^{\star}\right)\right)^{\star}\\
&=&\left(\dnr\left(\left(\dnr x\right)^{\star}\right)\right)^{\star}\\
&=&\left(\dnr\left(\upr x'\right)\right)^{\star}\\
&=&\left(\dnr x'\right)^{\star}\\
&=&\upr x''\\
&=&\upr x.
\end{eqnarray*}
Recall now that $(\mathord\uparrow x)\dgr=(\te(\upr x))^{\star}$.

Obviously, $\mathord\uparrow(x')=(\mathord\uparrow x)\dgr$. And therefore, by virtue of the reasoning above, we obtain that the bisemilattice $\mb L=\ty{L, \land, \lor, ',0,1}$ is isomorphic to the bisemilattice $\mb X=\ty{X, \cdot, +, \dgr, \bot, \top}$.
\end{proof}

We will now show that a rather mild addendum to the statements of the previous results will be adequate to comprehend the concept of involutive bisemilattice.
Consider $\theta, \, \theta^{-1}$ and $^{\star}$ as defined on page \pageref{eq:bal1}, and assume that the following condition is also satisfied:
\begin{equation}\label{eq:hey}
\forall A,B\in X:\; \theta(A\cap(\theta(A))\cmp)\subseteq\theta(A\cap B).
\end{equation}
Then, upon resorting the conventions on page \pageref{operations}, the following lemma holds:
\begin{lemma}\label{lem:suggerito1.2}
The structure $\mb X=\langle X, \cdot,+,\dgr, \bot, \top\rangle$ is in $\mc{IDBS}$.
\end{lemma}
\begin{proof}
From conditions \eqref{eq:bal1} and \eqref{eq:hey}:
\begin{eqnarray*}
 A\cdot(A\dgr+B)&=&A\cap (A^{\dgr}+B)\\
 &=&A\cap \theta^{-1}(\theta(\theta(A)^{\star})\cup \theta(B))\\
 &=&\theta^{-1}(\theta(A\cap \theta(A)^{\star})\cup \theta(A\cap B))\\
 &=&\theta^{-1}(\theta(A\cap B))\\
 &=&A\cap B\\ 
 &=&A\cdot B.
\end{eqnarray*}\end{proof}
In the case of Lemma \ref{lem:suggerito1.2}, we will refer to a structure of the form $\mb X=\langle X, \cdot,+,\dgr, \bot, \top\rangle$ as an \emph{{IDBS} of sets.}
Upon recalling that, for an algebra $\mb L$ in $\mc{IDBS}$, $X=\{\mathord\uparrow x:x\in L\},\,
Y=\{\mathord\downarrow x:x\in L\}$, where $\upr x$ and $\dnr x$ are as in page \pageref{upr-dnr-arrows}, the following theorem holds: 
\begin{theorem}\label{thm: rpr1}
Let $\mathbf L=\langle L, \land ,\lor , ',0,1\rangle$ in $\mc{IDBS}$, then the structure $\mb X=\langle X, \cdot,+,\dgr, \bot, \top\rangle$ is an IDBS of sets isomorphic to $\mb L$. 
\end{theorem}
\begin{proof}
If $\mb L\in \mc{IDBS}$, assuming all the conventions in the proof of Theorem \ref{thm: rpr}, we can now check condition \eqref{eq:hey},
\begin{eqnarray*}
\theta (\upr x\cap (\upr x)\dgr)\cup\theta(\upr x\cap \upr y) &=&\dnr (\upr x\cap \upr x')\cup\dnr(\upr x\cap \upr y)\\
&=&\dnr \upr(x\land x')\cup\dnr\upr( x\land  y)\\
&=& \dnr(x\land x')\cup\dnr( x\land  y)\\
&=&\dnr\big((x\land x')\lor( x\land  y)\big)\\
&=&\dnr\big(x\land ( x'\lor  y)\big)\\
&=&\dnr(x\land y)\\
&=&\dnr\upr(x\land y)\\
&=&\dnr(\upr x\cap\upr y)\\
&=&\theta(\upr x\cap\upr y).
\end{eqnarray*}
\end{proof} 

Let us remark that Theorem \ref{thm: rpr} can be easily adapted to cover the case of bounded bisemilattices.
In fact, set $X$ and $Y$ as in Theorem \ref{thm: rpr}, and require that they possess a greatest and a smallest element $\overline 1, \overline 0$, respectively. Moreover, let the mapping $\theta$ be as defined on page \pageref{eq:bal1}.\footnote{There is a slight but harmless abuse of language here. In fact, on page \pageref{eq:bal1} there is also a condition involving $^{\star}$, which is clearly unnecessary in this context.} Then, the following theorem holds:
\begin{theorem}\label{thm: rpr2}
An algebra $\mathbf L=\langle L, \land ,\lor , 0,1\rangle$ is in $\mc{BDBS}$ if and only if it is of the form $\mb X=\langle X, \cdot,+, \bot, \top\rangle$, where
\begin{eqnarray*}
\bot&=&\theta^{-1}(\overline0);\\
\top&=&\overline1.
\end{eqnarray*}
\end{theorem}
\begin{proof}
Immediate from Theorem \ref{thm: rpr} and Theorem \ref{thm: rpr1}.
\end{proof}
As a side remark to Theorem \ref{thm: rpr2}, let us observe that, since in the context of bounded bisemilattices no order-reversing involution is available, it is necessary to  require that the set $X$ possesses a greatest element $\overline 1$ -- which would not be otherwise expressible -- and then designate, in the structure $\mb X=\ty{X, \cdot,+,\bot, \top}$, $\top=\overline1$. For similar reasons,  we have to require that the set $Y$ possesses a smallest element $\ov0$, and, because the mapping $\te$ is completely arbitrary, we have to define $\theta^{-1}(\overline0)=\bot$.

{Let us observe that, in general, the mappings $\theta, \, \theta^{-1}$ do not form a Galois connection. They do in case the following condition is assumed:
\begin{equation}\label{eq:latt}
\theta^{-1}(\theta(A)\cup\theta(A\cap B))\subseteq A \subseteq A\cap \theta^{-1}(\theta(A)\cup\theta(B)).
\end{equation}
In fact, if condition \eqref{eq:latt} is satisfied, then the structure we obtain is a distributive lattice.}

Finally, we notice that the algebra $\mb 3$ in Example \ref{exmp:3} is isomorphic to the involutive bisemilattice $\ty{\{\upr0, \upr1/2,\upr1\}, \cdot, +,\dgr\bot,\top}$, where $\bot=\upr0$, $\top=\upr1$. Indeed, as it should be, $\upr0+(\upr1/2\cdot\upr0)=\upr1/2\not=\upr0$.

\section{Fspaces and semilattices}\label{sec: fspcs}
Taking up an idea from Dunn and Hartonas \cite{HD97}, the first notion we recall here is the concept of Fspace.
\begin{definition}\label{def:mFsp}
An \emph{Fspace} $X$ is a partially ordered Stone space such that 
\begin{enumerate}
\item For any $U=\{x_a:a\in A\}\subseteq X$ the greatest lower bound of $U$ exists.
\item $X$ has a subbasis of clopen sets $\mc S=\{X_i\}_{i\in I} \cup  \{\overline  X_j\}_{j\in I}$ indexed by some set $I$, such that for each $i\in I$, $X_i$ is a principal upper set, generated by a point
$x_i  \in X$.
\item The subset $\{x_i : i\in I \}\subseteq X$ is join-dense in $X$, that is: every point $x$ is the least
upper bound of the $x_i$ Õs it covers.
\item The collection $X^{*} = \{X_i\}_{i\in I}$ is closed under finite intersections.
\end{enumerate}
\end{definition}
Let us remark that $\overline X_{j}$ denote complements of $X_j$ in $X$, and $x_i$ in point (3) are generators of $X_i$.

It can be seen that, if $X$ is an Fspace, then the set $\{\overline  X_j\}_{j\in I}$ is closed under finite unions. It is also clear that the set $X^{*}$ in Definition \ref{def:mFsp} is a meet semilattice. Fspaces \label{def:Fsp-morf} morphisms are continuous mappings $f: Y\to X$ that preserve greatest lower bounds  in X and such that $f^{-1}$ maps $X^{*}$ into $Y^{*}$.\\

For reader's convenience, let us report here a lemma (see \cite[Lemma 2.2]{HD97}) which will be useful for the development of our discussion.
\begin{lemma}\label{lem:methom}
If $X, Y$ are Fspaces, and $f:Y\to X$ a morphism, then $f^{*}=f^{-1}:X^{*}\to Y^{*}$ is a meet semilattice homomorphism. 
\end{lemma}
The notion of Fspace is essential for the construction of dual frames of ortholattices and lattices, in general.
\begin{definition}\label{def:frame}
A \emph{frame} is a triple $\ty{X, \bot, Y}$, where $X,\, Y$ are sets, perhaps with some additional structure, and $\bot\subseteq X\times Y$ is a binary relation that induces a Galois connection $(\lambda,\rho)$ between $\wp(X)$ and $\wp(Y)$, defined for $U\in \wp(X)$ and $V\in\wp(Y)$ by:
\[
\lambda U=\{y\in Y:U\bot y\}\text{ and }\rho V=\{x\in X:x\bot V\},
\]
where $U\bot y$ means that for all $u\in U$: $u\bot y$ and dually for $x\bot V$.\footnote{For a set $Z$, the writing $\wp(Z)$ denotes its power set.}
\end{definition}

In \cite{Go75}, R. Goldblatt showed that, in the case of ortholattices, their dual frames are just two copies of the space of filters, while the (irreflexive and symmetric) relation $\bot$ is defined on the filter space by $x\bot y$ if and only if there is a $a\in x$ such that $\lnot a\in y$.
{As we will see, in the context of distributive bisemilattices it will be required to take into account \emph{both} the spaces of filters and ideals, since they capture possibly different semilattice orders.}

\section{A Stone-type representation for $\mc{DBS}$}\label{sec:top}
In order to provide a Stone-type duality for distributive bisemilattices, in this section, we combine ideas from section \ref{sec:rep} together with Hartonas and Dunn duality for lattices. In fact, we establish a new representation, combining extensions of Balbes' representation of distributive bisemilattices with  Hartonas and Dunn's results on lattices \cite{HD97}.

To this aim, we borrow from \cite{HD97} the idea of Fspace (see Definition \ref{def:mFsp}) and introduce the notion of 2space.
Here, the main role will be played by a bijective correspondence $\rho$ (satisfying a few opportune requirements) between the two Fspaces involved in the notion of 2space. As we will see, in general, the mapping $\rho$ could be quite arbitrary. {The idea behind our construction is rather simple. Consider a distributive bisemilattice $\mb L$. We put together the Fspaces associated to meet and join semilattice reducts of $\mb L$, and then use an adaptation of Balbes' conditions to connect the two spaces. This is somehow similar to what Hartonas and Dunn call a  ``$\bot$-frame'': a couple of Fspaces $X$ and $Y$ connected by a mapping $\bot$. However, it is our impression that we are moving along different degrees of freedom: in our case we are looking for a counterpart of distributivity laws, but not for absorption, in general. In their case, instead, in a $\bot$-frame the mapping $\bot$ induces a Galois connection. For the context of lattices, indeed, the Fspaces $X$ and $Y$ are supposed to reflect (dually) the same ordering. Instead, in a 2space the Fspaces involved may capture orders that are completely unrelated, and therefore the mapping $\rho$ could be as rowdy as possible.}
{None the less, in case the structure we start with is a distributive lattice, what we obtain from our construction are just two homeomorphic copies of the same space.}

In fact, it is perhaps worth stressing the fact that, if in the context of ortholattices and lattices the filter space and the ideal space capture, dually, one and the same order, this is no longer the case for bisemilattices, in general. As we have already mentioned, the filter space and the ideal space of a bisemilattice need not bear any relation with each other. This is a consequence of the fact that the notion of filter captures the meet-order $\leq_{\land}$, while the concept of ideal is relative to the join-order $\leq_{\lor}$. And, in a bisemilattice, the orders $\leq_{\land}$ and $\leq_{\lor}$ may be highly unrelated (the unique interaction coming from distributivity laws), because absorption, in general, may fail (see Example \ref{exmp:3}).

\begin{definition}\label{def:2sp}
A \emph{2space} is a triple $\langle X,\rho, Y\rangle$ such that:
\begin{enumerate}
 \item $X, Y$ are Fspaces with subbases indexed by a set $A$ (Definition \ref{def:mFsp});
 \item $\rho: X^{*}\to \overline Y^{*}$ is a bijective correspondence that satisfies conditions \eqref{eq:bal1} and \eqref{eq:un} on page \pageref{eq:bal1} and:
\begin{equation}\label{eq:rho}
 \rho(X_a)=\overline  Y_{a}.
\end{equation}
\end{enumerate}
\end{definition}
{As we mentioned, in general, $\rho,\, \rho^{-1}$ need not be polarities (an antitone Galois connection), and can be as arbitrary as possible. As regards their logical significance, the maps $\rho,\,\rho^{-1}$ connecting the Fspaces $X$ and $Y$ should not be considered sorts of ``generalized negations'', as in the case of polarities in a $\bot$-frame \cite{HD93}. The situation will be different in section \ref{sec:demor}, where a duality for $\mc{DDBS}$ will be discussed.}

Taking advatage of Definition \ref{def:2sp}, and (part of) the ideas from Theorem \ref{thm: rpr} we obtain that 

\begin{lemma}\label{lem:2sp-bsem}
 If $\langle X,\rho, Y\rangle$ is a 2space, the structure $\langle X^{*}, +,\cdot\rangle$ is a distributive bisemilattice, where, for all $Z,W\in X^{*}$
\begin{eqnarray*}
Z\cdot W&=&Z\cap W;\\
Z+W&=&\rho^{-1}(\rho(Z)\cup\rho(X)).
\end{eqnarray*}
 
\end{lemma}
\begin{proof}
By Definition \ref{def:mFsp}, $X^{*}$ and $\overline{Y}^{*}$ are meet and join semilattices, respectively. Then, by Definition \ref{def:2sp}-(2), $X^{*}$ is closed under $+$. We only prove associativity of $+$. It can be seen that, for $H,K,J\in X^{*}$, 
\begin{eqnarray*}
 (H+K)+J&=&\rho^{-1} (\rho\rho^{-1}(\rho(H)\cup\rho (K))\cup\rho (J))\\
& =&\rho^{-1} (\rho(H)\cup\rho (K)\cup\rho (J))\\
&=&\rho^{-1} (\rho(H)\cup\rho\rho^{-1}(\rho (K)\cup\rho (J)))\\
&=&H+(K+J)
\end{eqnarray*}
Finally, as in the proof of Theorem \ref{thm: rpr}, by virtue of conditions \eqref{eq:bal1} and \eqref{eq:un} both distributive laws obtain.
\end{proof}
Consider a distributive bisemilattice $\mb L$, and let $X=\mc F(L)$, $Y=\mc I(L)$, the sets of its filters and ideals, respectively. In order to keep the notation uniform, in the present, and in section \ref{sec:demor}, we will adopt the following linguistic conventions, for $a\in L$ we set 
\[
X_{a}=\{F\in X: a\in F\}, \qquad\overline X_{a}=\{F\in X: a\notin F\}, 
\]
and dually for $Y$.

\begin{theorem}\label{thm: toprepr}
Let $\mathbf L$ be a distributive bisemilattice. Set 
\[
\mathcal S= \{X_a\}_{a\in L} \cup  \{\overline  X_a\}_{a\in L}\mbox{ and }\mathcal P= \{Y_a\}_{a\in L} \cup  \{\overline  Y_a\}_{a\in L},
\]
as subbases for $X$ and $Y$, respectively, and define $\rho(X_a)=\overline  Y_{a}$. Then $\langle X,\rho, Y\rangle$ is a 2space.
\end{theorem}

\begin{proof}
The proof that $X$ and $Y$ are Fspaces are generalizations of the case of semilattices. For reader's convenience, we resume the major steps of the proof. Suppose that $F\not=G$ in $X$. Then, there is an $a\in L$ such that $a$ belongs to $ F$ but not to $G$. Therefore $F\in X_a$ and $G\in \overline  X_a$. Evidently, the set $X$ has a complete lattice structure, and the lattice is generated by its finitely generated members. As regards compactness, suppose that, for $P,Q\subseteq L$, $\{X_{p}:p\in P\}\cup \{\overline  X_{q}:q\in Q\}$ is a subbasic cover that does not admit any finite subcover. Let 
\[
F_{Q}=\{x\in L:\exists q_1,\ldots, q_n\in Q\,\big(\bigwedge_{i=1}^{n}q_i\leq_{\land} x\big)\}
\]
 be the filter generated by $Q$ in $L$. Obviously $F_{Q}\in \bigcap_{q\in Q}X_{q}$, and therefore $F_{Q}\notin \bigcup_{q\in Q}\overline  X_{q}$. We now show that $F_{Q}\cap P=\emptyset$.
Suppose that there is an $a\in L$ s.t. $a\in F_{Q}\cap P$. Because $F_{Q}$ is generated by ${Q}$, there are $q_{1},\ldots,q_{n}$ in $Q$ s.t. $\bigwedge_{i=1}^{n}q_i\leq_{\land} a$.  Since compactness fails, there is an $H\in X$ so that $H\notin \bigcup_{i=1}^{n} \overline  X_{q_{i}}\cup X_{a}$. But then, $\{q_i\}_{i\in I}\subseteq H$ and $a\notin H$. And so $\bigwedge_{i=1}^{n}q_i\not\leq_{\land} a$: a contradiction. Therefore, there is no $X_{p}$, for $p\in P$, such that $F_Q\in X_p$. Hence $\{X_{p}:p\in P\}\cup \{\overline  X_{q}:q\in Q\}$ is not a cover for $X$, because it does not contain $F_Q$. That $\rho$ is a bijective correspondence between $X^{*}$ and $\overline  Y^{*}$ derives from prime filter/ideal separation Theorem. Finally, conditions \eqref{eq:bal1} and \eqref{eq:un} follow from the proof of Theorem \ref{thm: rpr}.
\end{proof}
By Theorem \ref{thm: toprepr}, for brevity sake, if $\mb L$ is a distributive bisemilattice, we will denote by $S(\mb L)$ the 2space associated to $\mb L$.

We have seen in the section \ref{sec: fspcs} (cf. page \pageref{def:Fsp-morf}) the notion of Fspace morphism. We expand this idea to the framework of 2spaces.

\begin{definition}\label{def:2sp morph}
Let $\langle X,\rho, Y\rangle$ and $\langle Z,\rho, W\rangle$ be two 2spaces. A \emph{2space-morphism} (briefly, a \emph{morphism}) is a pair $(\psi,\chi)$ of morphisms, where $\psi:X\to Z,\,\chi:Y\to W$ are Fspace morphisms, such that the following diagram is commutative:
\begin{center}
 \begin{equation}\label{diag1}
  \xymatrix{
  {X^{*}}&&\ar@{->}^{\rho^{-1}}[ll]\overline  Y^{*}\\\\
Z^{*}\ar@{->}^{\psi^{-1}=\psi^{*}}[uu]\ar@{->}^{\sigma}[rr]&&\overline  W^{*}\ar@{->}^{\chi^{-1}=\chi^{*}}[uu]
}
\end{equation}
\end{center}
\end{definition}
From Definition \ref{def:2sp morph} we obtain:
\begin{lemma}\label{lem:morph trans}
If $\langle X,\rho, Y\rangle$ and $\langle Z,\sigma, W\rangle$ be 2spaces and $(\psi,\chi)$ a morphism between them, then $\psi^{*}:Z^{*}\to X^{*}$ is a distributive bisemilattice homomorphism.
\end{lemma}
\begin{proof}
By Lemma \ref{lem:methom} (see also \cite[Lemma 2.2]{HD97}), $\psi^{*}$ is a meet-semilattice homomorphism between $Z^{*}$ and $X^{*}$. It remains to be verified that $+$ is preserved by $\psi^{*}$. Indeed, for $H,K\in Z^{*}$, 
\begin{align*}
\psi^{*}(H+K)=&\psi^{*}\sigma^{-1}(\sigma(H)\cup\sigma(K))&(\tiny{\text{definition of }+})\\
=&\psi^{-1}\sigma^{-1}(\sigma(H)\cup\sigma(K))&(\tiny{\text{definition of }\psi^{*}})\\
=&\rho^{-1}\chi^{-1}\sigma\sigma^{-1}(\sigma(H)\cup\sigma(K))&(\tiny{\text{commutativity of diagram \eqref{diag1}}})\\
=&\rho^{-1}\chi^{-1}(\sigma(H)\cup\sigma(K))&
\end{align*}

and this is equal to $\rho^{-1}(\chi^{-1}\sigma(H)\cup\chi^{-1}\sigma(K))$. 
From the fact that $\chi$ is an Fspace morphism, it follows that $\chi^{-1}\sigma(H),\chi^{-1}\sigma(K)\in \overline Y^{*}$, and then $\chi^{-1}\sigma(H)\cup\chi^{-1}\sigma(K)\in \overline Y^{*}$. Thus, by Definition \ref{def:2sp}-(2), we obtain that $\rho^{-1}(\chi^{-1}\sigma(H)\cup\chi^{-1}\sigma(K))\in X^{*}$.
Upon noticing that, because of the commutativity of diagram \ref{diag1}, for $W\in Z^{*}$, 
\begin{eqnarray*}
\rho(\psi^{*}(W))&=&\rho(\psi^{-1}(W))\\
&=&\rho(\rho^{-1}(\chi^{*}(\sigma(W))))\\
&=&\rho(\rho^{-1}(\chi^{-1}(\sigma(W))))\\
&=&\chi^{-1}(\sigma(W))\\
&=&\chi^{*}(\sigma(W))
\end{eqnarray*}
 we have that $\rho^{-1}(\chi^{-1}\sigma(H)\cup\chi^{-1}\sigma(K))=\rho^{-1}(\rho(\psi^{*}(H))\cup\rho(\psi^{*}(K)))=\psi^{*}(H)+\psi^{*}(K)$, by the definition of $+$. 
\end{proof}
We now prove a converse of the previous statement:

\begin{lemma}\label{lem:morph trans 2}
Let $f:\mathbf L\to \mb M$ be a distributive bisemilattice homomorphism and $S(\mb L)=\langle \mc F(L), \rho, \mc I(L)\rangle$, $S(\mb M)=\langle \mc F(M), \sigma, \mc I(M)\rangle$ the associated 2spaces. Then $S(f)=(f^{-1},f^{-1}): S(\mb M)\to S(\mb L)$ is a 2space morphism.
\end{lemma}
\begin{proof}
Let us observe that $S(f)$ is an Fspace morphism with respect to both $\mc F(L)$ and $\mc I(L)$. For details we refer the reader to \cite[Lemma 2.5]{HD97}. As regards commutativity of diagram \ref{diag1}, upon noticing that, for any $a\in L$, $S(f)^{-1}(X_{a})=X_{f(a)}$ and $S(f)^{-1}(\overline  X_{a})=\overline  X_{f(a)}$, we compute: 
\begin{eqnarray*}
S(f)^{*}(X_{a})&=&S(f)^{-1}(X_{a})\\
&=&X_{f(a)}\\
&=&\rho^{-1} (\overline X_{f(a)})\\
&=&\rho^{-1} (S(f)^{-1}(\overline  X_{a}))\\
&=&\rho^{-1} (S(f)^{-1}(\sigma  (X_{a})))\\
&=&\rho^{-1} (S(f)^{*}(\sigma  (X_{a})))
\end{eqnarray*}
which is exactly our requirement.
\end{proof}
Combining what we have so far, together with \cite[Proposition 2.11]{HD97}, we obtain the following theorem:
\begin{theorem}\label{thm: oneway}
Every 2space is of the form $S(\mb L)$, for a distributive bisemilattice $\mb L$.
\end{theorem}
\begin{proof}
If $\langle X,\rho ,Y\rangle$ is a 2space over a set $Z$, with a subbase indexed by a set $A$, then $X,\,Y$ can be regarded, by virtue of the proof of \cite[Proposition 2.11]{HD97}, as the filter space of the meet semilattice $\langle A, \leq_{\land}\rangle$, where, for $a,b\in A$, $a\leq_{\land}b$ iff $X_{a}\subseteq X_{b}$, and $\langle A, \leq_{\lor}\rangle$, equipped with the order naturally induced by $Y$. Therefore, one can verify that $\langle X,\rho ,Y\rangle$ is the 2space $\langle \mc F(A),\sigma ,\mc I(A)\rangle$ arising from the distributive bisemilattice $\mb A=\langle A,\land, \lor\rangle$, where, clearly, $\sigma(X_{a})=\overline  Y_{a}$.
\end{proof}
Moreover,
\begin{lemma}\label{lem:uniq-morph}
If $(\psi,\chi)$ is a morphism between the 2spaces $S(\mb L)=\langle \mc F(L), \rho, \mc I(L)\rangle$ and $S(\mb M)=\langle \mc F(M), \sigma, \mc I(M)\rangle$, then there is a unique bisemilattice homomorphism $f:\mathbf M\to \mathbf L$ such that $(\psi,\chi)=(S(f),S(f))$ and, for any $a\in M$, $f(a)=b$ iff $\psi^{-1}(X_{a})=Y_{b}$ iff $\chi^{-1}(\overline  X_{a})=\overline  Y_{b}$.
\end{lemma}
\begin{proof}
By virtue of Theorem \ref{thm: oneway}, we can assume, without loss of generality, that the 2spaces are of the form $S(\mb L)=\langle \mc F(L), \rho, \mc I(L)\rangle$ and $S(\mb M)=\langle \mc F(M), \sigma, \mc I(M)\rangle$, where $L, M$ are the index sets for the subbases of the pairs of spaces $(\mc F(L), \mc I(L))$ and $(\mc F(M), \mc I(M))$, respectively. Set, for any $a\in M$, $\psi^{*}(a)=b$ iff $\psi^{-1}(X_{a})=Y_{b}$, and $\chi^{*}(a)=b$ iff $\chi^{-1}(\overline  X_{a})=\overline  Y_{b}$.
By Lemma \ref{lem:methom}, both $\psi^{*}$ and $\chi^{*}$ are semilattice homomorphisms. We now show uniqueness. Suppose, by way of contradiction, that $\psi\not=S({\psi^{*}})={\psi^{*}}^{-1}$. By total separation, for some index $a$, and $y$ in the space $\mc F(L)$, $S(\psi^{*}(y))={\psi^{*}}^{-1}(y)\in X_{a}$ and $\psi(y)\not\in X_{a}$. Then, $a\in {\psi^{*}}^{-1}(y)$, and so ${\psi^{*}}(a)\in y$. 
If ${\psi^{-1}}(X_{a})=Y_{b}$, set ${\psi^{*}}(a)=b$, by the definition of ${\psi^{*}}$.
Let us observe that, because ${\psi^{-1}}(X_{a})=Y_{b}$, we have that $y\in Y_{b}$. Therefore, $y\in \psi^{-1}(X_{a})$, which implies that $\psi(y)\in X_{a}$. A contradiction. Applying the same reasoning one can show that $\chi$ is uniquely determined by a semilattice homomorphism $g$. Now, by diagram \ref{diag1}, $Y_b=\psi^{-1}(X_{a})=\rho^{-1}(\chi^{-1}(\sigma(X_{a})))=\rho^{-1}(\chi^{-1}(\overline X_{a}))$. This implies that $\overline Y_{b}=\rho(Y_{b})=\rho(\rho^{-1}(\chi^{-1}(\overline X_{a})))=\chi^{-1}(\overline X_{a})$. This means that $f$ is equal to $g$. Therefore the 2space morphism $(\psi,\chi)$ is of the form $(S(f),S(f))$, with $f$ an homomorphism between the bisemilattices $\mb M$ and $\mb L$.
\end{proof}

Summarizing what we have so far we can state, as a corollary, the result announced for this section:
\begin{corollary}\label{cor:dual}
The categories of distributive bisemilattices and 2spaces, with respective morphisms, are dual. 
\end{corollary}

\section{Stone-type representations for $\mc{DDBS}$ and $\mc{IDBS}$}\label{sec:demor}
In this section, we enrich the notion of 2space and introduce the concept of 2space$^{\star}$
to extend the duality, established in section \ref{sec:top}, to the categories of De Morgan and involutive bisemilattices. We first notice that what we discuss in this section present evident similarities with the ideas in \cite[section 3]{HD93}.\footnote{The author was not aware of these results, and gratefully acknowledges an anonymous referee for calling his attention on this fact.} Indeed, a De Morgan operator on a bisemilattice $\mb A$ is an antitone Galois connection, in fact a dual equivalence, between the semilattice reducts. Therefore, implicitly, the ideas to cover case of bisemilattices with a De Morgan operator are sketched in \cite[section 3]{HD93}.
However, in order to complete the picture on the theme of this work (\emph{distributive} bisemilattices, and expansions thereof), we believe that it could be of some interest to elaborate on the ideas and techniques that we have developed in section \ref{sec:top} to cover the case of $\mc{DDBS}$, and then apply these results to the class of $\mc{IDBS}$.

It will be clear soon that our strategy will be, unimaginatively, along the lines of our Balbes-type representations of distributive De Morgan and involutive bisemilattices (cf. Theorem \ref{thm: rpr} and Theorem \ref{thm: rpr1}).

\begin{definition}\label{def:2sp*}
A \emph{\tsp} on an index set $A$ is a structure $\langle X,\rho,^\star, Y\rangle$, where $\langle X,\rho,Y\rangle$ is a 2space, the map $^\star:\overline  Y^{*}\to X^{*}$ is an order dual isomorphism such that $^{\star}\circ \rho=\rho^{-1}\circ{^{\star}}^{-1}$, and $\overline  Y^{*}$ possesses a least element $\overline0$.
\end{definition}
{The notion of \tsp is designed to capture the antitone Galois connection between the filter / ideal spaces of a De Morgan (distributive) bisemilattice. Indeed, the pair of mappings $(^{\star},^{\star^{-1}})$ reflects the fact that, in a De Morgan (distributive) bisemilattice, $x\leq_{\land} y$ iff $y'\leq_{\lor} x'$. Nevertheless, it is perhaps worth observing that, in general, it is not the case that $x\leq_{\land} y$ iff $y\leq_{\lor} x$. This is a major difference between bisemilattices (with involution) and lattices, in the latter case the filter / ideal spaces dually reflect the same order.
In this section, we mix the notion of Fspace and antitone Galois connection between Fspaces (cf. \cite[section 3]{HD93}), together with our Balbes' type representation of $\mc{DDBS}$ in order to obtain a full duality. More precisely, one may think of a \tsp as a $\bot$-frame \cite{HD93} enriched with a correspondence $\rho$, which we use, also in the case of $\mc{DDBS}$, to ``keep track'' of distributivity.
}

Taking advatage of Definition \ref{def:2sp*}, and the full strength of Theorem \ref{thm: rpr} we readily have that to any 2space$^{\star}$ is associated an algebra in $\mc{DDBS}$.
\begin{lemma}\label{lem:2sp-demor}
 If $\langle X,\rho, ^{\star},Y\rangle$ is a \tsp, the structure $\langle X^{*}, +,\cdot,\dgr,\bot,\top\rangle$ is a De Morgan bisemilattice, where the reduct $\langle X^{*}, +,\cdot,\rangle$ is as in Lemma \ref{lem:2sp-bsem}, and for $A\in X^{*}$
\begin{eqnarray*}
A\dgr&=&(\rho{(A)})\cmp;\\
\bot&=&\rho^{-1}(\overline0);\\
\top&=&\overline0^\star.
\end{eqnarray*}
\end{lemma}
\begin{proof}
The fact that the reduct $\langle X^{*}, +,\cdot,\rangle$ is a distributive bisemilattice follows from Lemma \ref{lem:2sp-bsem}. That the structure $\langle X^{*}, +,\cdot,\dgr,\bot,\top\rangle$ is De Morgan is a consequence of the first part of the proof of Theorem \ref{thm: rpr}. 
\end{proof}
Following the pattern of section \ref{sec:top}, we now present a converse of Lemma \ref{lem:2sp-demor}, that allows to associate to any $\mb L$ in $\mc{DDBS}$ a 2space$^{\star}$.

Let $\mathbf L$ be a De Morgan bisemilattice, and recall that
\[
\mathcal S= \{X_a\}_{a\in L} \cup  \{\overline  X_a\}_{a\in L}\mbox{ and }\mathcal P= \{Y_a\}_{a\in L} \cup  \{\overline  Y_a\}_{a\in L},
\]
are the subbases for $X$ and $Y$, respectively, and $\rho(X_a)=\overline  Y_{a}$ is as in Theorem \ref{thm: toprepr}. Upon setting, for $\overline {Y_{a}}\in \mathcal P$, $(\overline {Y_{a}})^{\star}=X_{a'}$, and $\overline0=\overline {Y_{0}}$, the following theorem holds:

\begin{theorem}\label{thm: toprepr2}
The structure $S(\mathbf L)=\langle X,\rho, ^{\star},Y\rangle$ is a \tsp.
\end{theorem}
\begin{proof}
The fact that $\langle X,\rho, Y\rangle$ is a 2space derives from Theorem \ref{thm: toprepr}. Furthermore, the mapping $^{\star}:\overline  Y^{*}\to X^{*}$, defined, for all $a\in L$, by $\ov{Y_{a}}^{\star}=X_{a'}$ is bijective because the operation $'$ is an involution. Moreover, $(\ov{Y_{a}}\cup \ov{Y_{b}})^{\star}=\ov{Y_{a\lor b}}^{\star}=X_{(a\lor b)'}=X_{a'\land b'}=X_{a'}\cap X_{b'}$, which shows that $^{\star}$ is order dual. The fact that $\ov {Y_{0}}$ is the least element of $\ov{Y^{*}}$ follows from (the dual of) Definition \ref{def:fil.id}: any ideal is a non-void subset of $L$ that, consequently, contains the element $0$. Finally, $(\rho(X_{a}))^{\star}=\ov{Y_{a}}^{\star}=X_{a'}=\rho^{-1}(\ov{Y_{a'}})=\rho^{-1}(X_{a'}^{\star^{-1}})$.
\end{proof}
With a slight, but harmless, conflict of notation with section \ref{sec:top}, if $\mb L$ is a De Morgan bisemilattice, we will denote (within this section) by $S(\mb L)$ the 2space$^{\star}$ associated to $\mb L$ (see Theorem \ref{thm: toprepr2}).\\

We have introduced in section \ref{sec:top} (precisely, on page \pageref{def:2sp morph}) the notion of 2space morphism. To properly match with the context of 2spaces$^{\star}$, in the following definition, we expand this notion with a natural additional requirement on the operation $^{\star}$.

\begin{definition}\label{def:2mor*}
Let $\langle X,\rho, ^{\star_{Y}},Y\rangle$ and $\langle Z,\sigma,^{\star_{W}}, W\rangle$ be 2spaces$^{\star}$. A \emph{\tsp-morphism} (briefly, a \emph{morphism}) $(\psi,\chi)$ is a 2space morphism (see Definition \ref{def:2sp morph}) that satisfies the following further condition, for all $A\in Z^{*}$:
\begin{equation}\label{cond:star}
\chi^{-1}(A{^{\star_{W}}}^{-1})=(\psi^{-1}(A)){^{\star_{Y}}}^{-1}\tag{$\star$},
\end{equation}
or, equivalently, the next diagram commutes:
\begin{center}
 \begin{equation}\label{diag-star}
  \xymatrix{
  {X^{*}}&&\ar@{<-}^{\star_{Y}^{-1}}[ll]\overline  Y^{*}\\\\
Z^{*}\ar@{->}^{\psi^{-1}=\psi^{*}}[uu]\ar@{->}^{\star_{W}^{-1}}[rr]&&\overline  W^{*}\ar@{->}^{\chi^{-1}=\chi^{*}}[uu]
}
\end{equation}
\end{center}

\end{definition}

Using Definition \ref{def:2mor*} we can show that:
\begin{lemma}\label{lem:morph trans3}
If $\langle X,\rho, ^{\star_{Y}},Y\rangle$ and $\langle Z,\sigma,^{\star_{W}}, W\rangle$ are 2spaces$^{\star}$ and $(\psi,\chi)$ a morphism between them, then $\psi^{*}=\psi^{-1}:Z^{*}\to X^{*}$ is a De Morgan bisemilattice homomorphism.
\end{lemma}
\begin{proof}
From Lemma \ref{lem:morph trans}, it follows that the mapping $\psi^{*} $ is a bisemilattice homomorphism between the reducts $\langle Z^{*}, +,\cdot,\rangle$ and $\langle X^{*}, +,\cdot\rangle$. As regards $\dgr$, if $A\in Z^{*}$, we have that 
\begin{align*}
\psi^{*}(A\dgr)&=\psi^{-1}((\sigma(A))^{\star_{W}})&\\
&= \rho^{-1}(\chi^{-1}(\sigma\big(((\sigma(A))^{\star_{W}})\big)))&\text{\tiny{(diagram \ref{diag1})}}\\
&= \rho^{-1}(\chi^{-1}(\sigma(\sigma^{-1}(A^{\star_{W}^{-1}}))))& \text{\tiny{(Definition \ref{def:2sp*})}}
\\
&= \rho^{-1}(\chi^{-1}(A^{\star_{W}^{-1}}))&\\
&= \rho^{-1}(\big(\psi^{-1}(A)\big)^{\star_{Y}^{-1}})&\text{\tiny{(condition \eqref{cond:star})}}\\
&=(\rho(\psi^{-1}(A)))^{\star_{Y}}&\text{\tiny{(Definition \ref{def:2sp*})}}\\
&=(\psi^{*}(A))\dgr&
\end{align*}
Finally, for the constants, $\psi^{*}$ preserves $\top$, because, by Lemma \ref{lem:methom}, it is a meet-semilattice homomorphism. Therefore, since $\top\dgr=({\ov0}^{\star})\dgr=(\rho({\ov0}^{\star}))^{\star}=\rho^{-1}(({\ov0}^{\star})^{\star^{-1}})=\rho^{-1}({\ov0})=\bot$, and, by the argument above, $\psi^{*}$ commutes with $\dgr$, we have that $\psi^{*}$ preserves also $\bot$.
\end{proof}
We now discuss a converse of the previous statement:
\begin{lemma}\label{lem:morph trans4}
Let $f:\mathbf L\to \mb M$ be a De Morgan bisemilattice homomorphism and $S(\mb L)$, $S(\mb M)$ the 2spaces$^{\star}$ associated to $\mb L$ and $\mb M$, respectively. Then $S(f)=(f^{-1},f^{-1}): S(\mb M)\to S(\mb L)$ is a \tsp morphism.
\end{lemma}
\begin{proof}
By virtue of Lemma \ref{lem:morph trans 2}, we only need to prove that condition \eqref{cond:star} is satisfied. Indeed, for $a\in L$, we have that 
\begin{eqnarray*}
\chi^{-1}((X_{a})^{\star^{-1}})&=&\chi^{-1}(\overline{Y_{a'}})\\
&=&\overline{Y_{f(a')}}\\
&=&\overline{Y_{f(a)'}}\\
&=&(X_{f(a)'})^{\star^{-1}}\\
&=&(\psi^{-1}(X_{a}))^{\star^{-1}}
\end{eqnarray*}
\end{proof}

From the results we have discussed so far in this section, together Theorem \ref{thm: oneway}, we obtain the following theorem:
\begin{theorem}\label{thm: oneway1}
Every 2space$^{\star}$ is of the form $S(\mb L)$, for a De Morgan bisemilattice $\mb L$.
\end{theorem}
\begin{proof}
If $\langle X,\rho ,^{\star},Y\rangle$ is a 2space$^{\star}$ over a set $Z$, with a subbase indexed by a set $A$, then $X,\,Y$ can be regarded, by Theorem \ref{thm: oneway}, as the filter space of the meet semilattices $\langle A, \leq_{\land}\rangle$ and $\langle A, \leq_{\lor}\rangle$, respectively. Upon setting $a'$ as the index uniquely associated to $(\rho(X_{a}))^{\star}$, and $0,1$ the indexes associated to $\rho^{-1}(\overline0)$ and $\overline0^\star$, respectively, it is immediate to verify that $\langle X,\rho,^{\star} ,Y\rangle$ is the 2space $\langle \mc F(A),\rho,^{\star} ,\mc I(A)\rangle$ arising from the De Morgan bisemilattice $\mb A=\langle A,\land, \lor,',0,1\rangle$.
\end{proof}

\begin{lemma}\label{lem:uniq-morph2}
If $(\psi,\chi)$ is a morphism between the 2spaces$^{\star}$ $S(\mb L)=\langle \mc F(L), \rho,^{\star}, \mc I(L)\rangle$ and $S(\mb M)=\langle \mc F(M), \sigma,^{\star}, \mc I(M)\rangle$, then there is a unique  homomorphism $f:\mathbf M\to \mathbf L$ between the De Morgan bisemilattices $\mb M$ and $\mb L$ such that $(\psi,\chi)=(S(f),S(f))$.
\end{lemma}
\begin{proof}
Let us observe that the homomorphism $f$, between the De Morgan bisemilattices $\mb M$ and $\mb L$, induced by the 2space$^{\star}$ morphism $(\psi,\chi)$ is, in particular, an homomorphism between bisemilattices. Therefore, the claim follows from Lemma \ref{lem:uniq-morph}.
\end{proof}
As a consequence of the results in this section, we can state, as a corollary, the following:

\begin{corollary}\label{cor:dual2}
The categories of De Morgan bisemilattices and 2spaces$^{\star}$, with respective morphisms, are dual.\end{corollary}
Let us close this section by noticing that the machinery we have presented so far can be extended with ease to cover the case of $\mathcal{IDBS}$. Indeed, it will be enough to require in the definition of 2space$^{\star}$ $\langle X,\rho, ^{\star}, Y\rangle$ that, for all $H,K\in X^{*}$:
\begin{equation}\label{eq:last}
 \rho(H\cap(\rho(K))\cmp)\subseteq\rho(H\cap K).
\end{equation}
In fact, condition \eqref{eq:last} is nothing but the translation into the present context of condition \ref{eq:hey} in Theorem \ref{thm: rpr1}, which provides the Balbes' type representation of involutive bisemilattices.

\subsection*{Acknowledements}
The author gratefully acknowledges the support of the Horizon 2020 program of the European Commission: SYSMICS project, Proposal Number: 689176, MSCA-RISE-2015, the support of the Italian Ministry of Scientific Research (MIUR) within the FIRB project ``Structures and Dynamics of Knowledge and Cognition'', Cagliari, Proposal Number: F21J12000140001, and the support of Fondazione Banco di Sardegna within the project ``Science and its Logics: The Representation's Dilemma'', Cagliari, Proposal Number: F72F16003220002.
\\ The writer wishes to express his gratitude to Francesco Paoli for calling his attention to the topic of this article, and Jos\'e Gil-F\'erez for the insightful discussions. Finally, the author feels deeply indebted to Heinrich Wansing and to the anonymous Studia Logica's reviewers for their careful advices, which helped a lot to improve this article.

\end{document}